\documentclass[a4paper,11pt,reqno]{amsart}
\usepackage{amsmath, amsthm, amssymb, mathrsfs, amsfonts}
\usepackage{array}
\usepackage[active]{srcltx}
\usepackage{graphicx}
\usepackage[margin=3cm]{geometry}
\usepackage[english]{babel}
\usepackage[colorlinks=true, allcolors=blue]{hyperref}
\usepackage{enumerate}
\usepackage{float}

\theoremstyle{definition}
\newtheorem{defin}{Definition}[section]
\newtheorem{ejem}{Example}[section]

\theoremstyle{plain}
\newtheorem{teor}{Theorem}[section]
\newtheorem{lem}{Lemma}[section]
\newtheorem{prop}{Proposition}[section]
\newtheorem{cor}{Corollary}[section]

\def\col{\mathop{\rm col}\nolimits}

\DeclareMathOperator{\wt}{wt}

\usepackage{tikz,xcolor,hyperref}

\definecolor{lime}{HTML}{A6CE39}
\DeclareRobustCommand{\orcidicon}{%
	\begin{tikzpicture}
	\draw[lime, fill=lime] (0,0)
	circle [radius=0.16]
	node[white] {{\fontfamily{qag}\selectfont \tiny ID}};
	\draw[white, fill=white] (-0.0625,0.095)
	circle [radius=0.007];
	\end{tikzpicture}
	\hspace{-2mm}
}

\foreach \x in {A, ..., Z}{%
	\expandafter\xdef\csname orcid\x\endcsname{\noexpand\href{https://orcid.org/\csname orcidauthor\x\endcsname}{\noexpand\orcidicon}}
}

\title{$S_h$-sets and linear codes over $\mathbb{F}_q$}

\author[V. Guerrero]{Viviana Carolina Guerrero Pantoja\orcidA{}}
\address{Viviana Guerrero, Departamento de Matem\'aticas, Universidad del Cauca}
\email{vivianagp@unicauca.edu.co}

\author[J.H. Castillo]{John H. Castillo\orcidB{}}
\address{John H. Castillo, Departamento de Matem\'aticas y Estad\'istica, Universidad de Nari\~no}
\email{jhcastillo@udenar.edu.co}

\author[C.A. Trujillo]{Carlos Alberto Trujillo Solarte\orcidC{}}
\address{Carlos Alberto Trujillo Solarte, Departamento de Matem\'aticas, Universidad del Cauca}
\email{trujillo@unicauca.edu.co}


\begin{document}

 \noindent
\bibliographystyle{plain}

\renewcommand{\refname}{\textbf{References}}
\thispagestyle{empty}


\begin{abstract}
Let $(G,+)$ be an Abelian group. Given $h\in \mathbb{Z}^+$, a non-empty subset $A$ of $G$ is called an $S_h$-set if all the sums of $h$  distinct elements of $A$ are different. We extend the concept of $S_h$-set to a more general context in the setting of finite vector spaces over finite fields. More precisely, $\emptyset \neq A\subseteq \mathbb{F}_q^r$ is called an $S_h$-linear set if all linear combinations of $h$ elements of $A$ are different. We establish a correspondence between $q$-ary linear codes and $S_h$-linear sets. This connection allow us to find lower bounds for the maximum size of $S_h$-sets in $\mathbb{F}_q^r$.
\end{abstract}

\maketitle

\section{Introduction}
The origin of coding theory is closely tied to the seminal work of the American electrical engineer, mathematician and computer scientist Claude E. Shannon; his article ``A mathematical theory of communication" \cite{SH} originated both coding theory and information theory. The main objective of error-correcting codes is to construct codes that allow the transmission of the maximum possible information, detect errors produced during transmission, and correct them. However, Shannon's article did not contain an explicit construction of such codes. Richard W. Hamming \cite{hamming1950error} and Marcel Golay \cite{Golay49} were potentially the pioneers in providing explicit formulations of codes. Since then, various mathematical techniques have been employed for this purpose, resulting in different families of codes, such as block codes, linear codes, cyclic codes, among others. Several branches of mathematics are involved in the construction and study of these codes, including linear algebra, group theory, ring theory, finite fields theory, module theory, combinatorics, number theory, etc.

The relationship between coding theory and additive number theory was first proposed by R.L. Graham and N.J.A. Sloane in 1980 in their article ``Lower bounds for constant weight codes" \cite{GS}, where they relate $S_{h}$-sets and constant weight binary codes. In 1999, G. Cohen and G. Zémor, in  ``Subset and Coding Theory" \cite{CZ}, presented how coding theory techniques can be used to solve problems in additive number theory. This is achieved by associating a linear code $C(S)\subseteq \mathbb{F}_2^{n}$ with a generating set $S$ of $\mathbb{F}_2^{r}$ with $|S|=n$. Through this relationship, they demonstrate how four additive problems in the Abelian group 
$\mathbb{F}_2^{r}$ can be expressed as coding theory problems, and using their techniques, they present original contributions. In \cite{CZ}, it is mentioned that not every code $\mathcal{C}$ is necessarily a code $C(S)$ for some set $S$; more precisely, they stated that for any code $\mathcal{C}$, there exists a set $S$ that does not contain the zero vector $\boldsymbol{0}$ such that $\mathcal{C}=C(S)$ if and only if the minimum distance of $\mathcal{C}$ is greater or equal than $3$.

Later, G. Cohen, S. Litsyn, and G. Z\'emor, in \cite{CLZ}, study the $S_{2}$-sets in $\mathbb{F}_{2}^{r}$ using an associated code to determine the maximum number of elements that an $S_{2}$-set can have. Subsequently, H. Derksen in \cite{D} revisited the ideas addressed in \cite{GS}, constructing new constant weight binary codes and from them new non-linear binary codes.

Thereafter, H. Haanp\"{a}\"{a} and P. \"{O}sterg\'ard in \cite{HO} demonstrate a one-to-one correspondence between $[n,n-r,5]$-binary linear codes and $S_{2}$-sets  of size $n+1$ in $\mathbb{F}_{2}^{r}$. Afterward, C. Gómez and C. Trujillo in \cite{GT} generalize the result of \cite{HO}, extending the correspondence to $S_{h}$-sets. More precisely, they proved that there exists an $[n,k,d]$-binary linear code with $d\geq 2h+1$ if and only if there exists an $S_h$-set with size $n+1$ in $\mathbb{F}_{2}^{n-k}$, where $2h\leq n-k$.   Recently, I. Czerwinski and A. Pott in \cite{CP2024} revisited the ideas of G. Cohen and G. Zémor in \cite{CZ} and demonstrated a result equivalent to the one shown in \cite{HO}.

This article is organized as follows. In the first section, we give some definitions, properties, and we recall some well known results in coding theory. In the second one, we introduce the concept of $S_h$-linear set in finite vector spaces. We give some properties and examples of $S_h$-linear sets and prove that this is a natural extension of the concept of $S_h$-set. Moreover, in this section we give our main result, see Theorem \ref{principalteor}. Finally, in the last section, we give some consequences of our mains results.
 
\section{Preliminaries}
Let $\mathbb{F}_q$ be the finite field with $q$ elements, where $q$ is a power of a prime number and $\mathbb{F}_q^n$ denote the vector space of all $n$-tuples over $\mathbb{F}_q$. An \emph{$[n,k]$-linear code} $\mathcal{C}$ is an $k$-dimensional subspace of $\mathbb{F}_q^n$. We will say that $n$ is the length of $\mathcal{C}$.  It is used to say that $\mathcal{C}$ is a \emph{$q$-linear code}, in particular when $q=2$ or $q=3$ the code is called \emph{binary} or \emph{ternary} linear code, respectively. An element of a $q$-linear code is called a \emph{codeword}. The \emph{Hamming distance}, $d(\boldsymbol{x},\boldsymbol{y})$, between two codewords $\boldsymbol{x}=(x_1,\ldots,x_n), \boldsymbol{y}=(y_1,\ldots,y_n)\in \mathcal{C}\subseteq\mathbb{F}_q^n$ is the number of entries where they differ, or equivalently, $d(\boldsymbol{x},\boldsymbol{y}) =|\{i:x_i\neq y_i, ~1\leq i\leq n\}|$.

For $\boldsymbol{x}\in\mathbb{F}_q^n$, the \emph{Hamming weight} of $\boldsymbol{x}$ is $\wt(\boldsymbol{x})=d(\boldsymbol{x},\boldsymbol{0})$, i.e., $\wt(\boldsymbol{x})$ is the number of non-zero coordinates in $\boldsymbol{x}$. The \emph{minimum distance} $d(\mathcal{C})=d$ of a linear code $\mathcal{C}$ is defined as the minimum weight among all non-zero codewords, thus we called it an \emph{$[n,k,d]$-linear
code}. A \emph{generator matrix} for an $[n, k]$-linear code $\mathcal{C}$ is any $k\times n$ matrix $G$ whose rows form a basis of the vector subspace $\mathcal{C}$. Thus, the code $\mathcal{C}$ can be seen as $\mathcal{C}=\{ \boldsymbol{x}G: \boldsymbol{x}\in \mathbb{F}_q^k\}$.

Also, as an $[n,k]$-linear code is a subspace of a vector space, it is the kernel of a linear transformation. Hence,
there exists an $(n-k)\times n$ matrix $H$, called a \emph{parity-check matrix} for the $[n, k]$-linear code $\mathcal{C}$, such that $\mathcal{C} =\{\boldsymbol{x} \in \mathbb{F}_q^n:  H\boldsymbol{x}^T = \boldsymbol{0}\}$. We recall, without proof, some classical results that we will use later. 

\begin{teor}[{\cite[Cor. 4.5.7]{LX}}]\label{cojunto-LI-H}
If $H$ is a parity-check matrix of a code $\mathcal{C}$ with length $n$, then $\mathcal{C}$ has minimum distance $d$ if and only if any set with $d-1$ columns of $H$ is a linearly independent set, and there exists a set with $d$ columns of $H$ that is linearly dependent.
\end{teor}
The next result established a relation between the parameters of a linear code, it is known as Singleton bound, see \cite[Thm. 5.4.1]{LX}.
\begin{teor}[Singleton bound]\label{CotaSingleton}
If $\mathcal{C}$ is an $[n, k, d]$-linear code, then $k + d \leq n + 1$.
\end{teor}

Let $\langle G, +\rangle$ be an Abelian group and $h\in \mathbb{Z}^+$. A subset $S$ of $G$, where $|S| = k$, is an $S_h$-set of size $k$ if all sums of $h$ different elements in $S$ are distinct in $G$, i.e., if all the expressions $x_{i_1}+x_{i_2}+\cdots+x_{i_h},$
with $i_1<i_2<\cdots<i_h$ and $x_{i_1},x_{i_2},\ldots,x_{i_h}\in S$, generate different elements of $G$. It is clear that every subset of $G$ is an $S_1$-set.  Besides, a non-empty subset $S$ of an Abelian group $(G,+)$ is a \emph{Sidon set} see \cite{Sidon1932,sidon1935fourier} or a \emph{$B_2$-sequence}, see \cite{Sidon1932,sidon1935fourier}, if $a + b = c + d, a, b, c, d \in S$ imply $\{a, b\} = \{c, d\}$.  Note that in the concept of Sidon sets, repetitions in the terms of the sum are allowed, unlike in an $S_2$-set. However, when defined over $\mathbb{F}_2$, the notions of a Sidon set and an $S_2$-set are equivalent, since in this field repetitions are not possible.

\section{A correspondence between $q$-linear codes and $S_h$-sets}\label{sec3}

In this section,  we obtain a generalization to finite fields $\mathbb{F}_q$ ($q\geq 2$) of a relation between $S_h$-sets and binary linear codes given by C. Gómez and C. Trujillo in \cite{GT}. Firstly, we introduce the concept of $S_{h}$-linear set and prove some of its properties. In particular, it is established that a linear code $\mathcal{C}$ with $d(\mathcal{C})\geq 2h+1$ over $\mathbb{F}_q$ is associated with an $S_{h}$-set in $\mathbb{F}_{q}^{r}$. 

\begin{defin}[$h$-linear combination]\label{def:h_linear}
	Let $A$ be a non-empty subset of a finite vector space $V$ over  $\mathbb{F}_q$ and $h\leq |A|$ be a positive integer. An \emph{$h$-linear combination of $A$} is a linear combination of $h$ distinct elements from $A$. In other words, an  $h$-linear combination of $A$ is an expression of the form 
	\begin{equation}\label{eq:h_linear}
		\lambda_{1}\boldsymbol{a}_{1}+\lambda_{2}\boldsymbol{a}_{2}+\dots+\lambda_{h}\boldsymbol{a}_{h}, \text{ where } \lambda_{i}\in \mathbb{F}_q^*\text{ and } \boldsymbol{a}_{i}\in A.
	\end{equation}
\end{defin}

\begin{defin}[$S_{h}$-linear set]\label{sh_linear_set}
	Let $A$ be a non-empty subset of a finite vector space $V$ over $\mathbb{F}_q$ and $h\leq |A|$ be a positive integer. We say that $A$ is an \emph{$S_{h}$-linear set on $V$}, if all $h$-linear combinations of elements from $A$, omitting permutations of the summands, yield distinct elements in $V$. In other words, $A$ is an $S_{h}$-linear set if all expressions of the form
	\begin{equation}\label{eq:h_linear_comb}
		\lambda_{1}\boldsymbol{x}_{i_1} + \lambda_{2}\boldsymbol{x}_{i_2} + \cdots + \lambda_{h}\boldsymbol{x}_{i_h}, \, \text{ with } \, i_1 < i_2 < \cdots < i_h,
	\end{equation}
	where $\lambda_{1},\dots,\lambda_{h}\in \mathbb{F}_q^*$ and $\boldsymbol{x}_{i_1},\,\boldsymbol{x}_{i_2} ,\, \ldots ,\, \boldsymbol{x}_{i_h} \in A$, produce distinct elements in $V$. 
\end{defin}
In this manuscript, we consider only the study of $S_h$-linear sets on the vector space $\mathbb{F}_q^r$. Note that if all scalars in the expression \eqref{eq:h_linear}  are equal to 1, we obtain the concepts of weak $h$-sum and $S_{h}$-set studied by C. G\'omez and C. Trujillo in \cite{GT}. Since that if $A$ is an $S_h$-linear set, all the sums (with coefficients equal to 1) of $h$ elements of $A$ are also different, then $A$ is also an $S_h$-set. Observe that these two concepts are equivalent when the scalars are taking from $\mathbb{F}_2$. However, it is not true on $\mathbb{F}_q$, with $q\neq 2$.

\begin{ejem}
 The set $S=\{(2,0,0,0,0),(1,2,1,1,0),(2,2,1,2,1),(0,0,0,2,2)\}$ is an $S_{2}$-set in $\mathbb{F}_{3}^{5}$, but it is not an $S_2$-linear set, since that
 $$(2,0,0,0,0)+2(1,2,1,1,0)=2(2,2,1,2,1)+2(0,0,0,2,2).$$
\end{ejem}

\begin{lem}\label{lema:LIset}
	Let $A$ be a non-empty subset of a finite vector space $V$ over a field $\mathbb{F}_q$. If $A$ is a linearly independent set, then $A$ is an $S_h$-linear set,  for all $1\leq h\leq |A|$.
\end{lem}
\begin{proof}
	Assume that there are two $h$-linear combinations of $A$ that are equal in $V$, i.e.,
	\begin{equation*}
		\lambda_{1}\boldsymbol{a}_{1}+\lambda_{2}\boldsymbol{a}_{2}+\dots+\lambda_{h}\boldsymbol{a}_{h}=\beta_{1}\boldsymbol{b}_{1}+\beta_{2}\boldsymbol{b}_{2}+\dots+\beta_{h}\boldsymbol{b}_{h},
	\end{equation*}
	hence
	\begin{equation}\label{eq:lemliset-shset}
		\lambda_{1}\boldsymbol{a}_{1}+\lambda_{2}\boldsymbol{a}_{2}+\dots+\lambda_{h}\boldsymbol{a}_{h}-\beta_{1}\boldsymbol{b}_{1}-\beta_{2}\boldsymbol{b}_{2}-\dots-\beta_{h}\boldsymbol{b}_{h}=\boldsymbol{0},
	\end{equation}
	We study the following cases:
	\begin{enumerate}[\normalfont(i)]  
		\item $\boldsymbol{a}_{i}=\boldsymbol{b}_{i}$ for all $i$. From \eqref{eq:lemliset-shset} 
		\begin{equation*}
			(\lambda_{1}-\beta_{1})\boldsymbol{a}_{1}+(\lambda_{2}-\beta_{2})\boldsymbol{a}_{2}+\dots+(\lambda_{h}-\beta_{h})\boldsymbol{a}_{h}=\boldsymbol{0}.
		\end{equation*}
		Since $A$ is a linearly independent set, this implies that $\lambda_i=\beta_i$, which is a contradiction.
		\item $\boldsymbol{a}_{i}\neq \boldsymbol{b}_{i}$ for all $i$. Again, from \eqref{eq:lemliset-shset} and our assumption we get that $\lambda_i=\beta_i=0$, an impossible consequence.
		\item Assume that $\boldsymbol{a}_{i}=\boldsymbol{b}_{i}$ for some $i$. Let $\emptyset\neq I=\{i:\boldsymbol{a}_{i}=\boldsymbol{b}_{i} \}\subsetneq \{1,\ldots, h\}$. Then \eqref{eq:lemliset-shset} can be transformed in 
		\begin{equation*}
			\sum_{i\in I}(\lambda_{i}-\beta_{i})\boldsymbol{a}_{i}+\sum_{j\notin I}(\lambda_{j}\boldsymbol{a}_{j}-\beta_{j}\boldsymbol{b}_{j})=\boldsymbol{0}.
		\end{equation*}
		As $A$ is linearly independent, from the last expression we get that $\lambda_j=\beta_j=0$ for all $j\notin I$, again a contradiction of Definition \ref{def:h_linear}.
	\end{enumerate}
	Therefore, $A$ is an $S_h$-linear set. 
\end{proof}
It should be noted that, although $A$ may be an $S_h$-linear set for some $h$, this does not implies that $A$ is linearly independent. In the sequel, denote with $\boldsymbol{e}_i$ the canonical vector of $\mathbb{F}_q^n$. Let $h\geq 2$ be a positive integer, we set $\overline{h}A$ the set of all $h$-linear combinations of $A$, i.e.,
\begin{equation}
\overline{h}A=\{\lambda_{1}\boldsymbol{a}_{1}+\lambda_{2}\boldsymbol{a}_{2}+\dots+\lambda_{h}\boldsymbol{a}_{h}: \lambda_{i}\in \mathbb{F}_q^*\text{ and } \boldsymbol{a}_{i}\in A\}.
\label{eq:}
\end{equation}

\begin{ejem}\quad
\begin{enumerate}
\item The converse of Lemma \ref{lema:LIset} is not true. In fact, for $$A=\{(0,0,0),(1,1,0),(0,1,0)\}\subset\mathbb{F}_3^3,$$ we obtain that
\begin{align*}
\overline{2}A&=\{(0,0,0)+(1,1,0),(0,0,0)+(0,1,0),(1,1,0)+(0,1,0),\\
&\qquad(0,0,0)+2(1,1,0),(0,0,0)+2(0,1,0),2(1,1,0)+(0,1,0),\\
&\qquad(1,1,0)+2(0,1,0),2(1,1,0)+2(0,1,0)\}\\
&=\{(1,1,0),(0,1,0),(1,2,0),(2,2,0),(0,2,0),(2,0,0),(1,0,0),(2,1,0)\},\\
\overline{3}A&=\{(0,0,0)+(1,1,0)+(0,1,0),(0,0,0)+2(1,1,0)+(0,1,0),\\
&\qquad(0,0,0)+(1,1,0)+2(0,1,0),(0,0,0)+2(1,1,0)+2(0,1,0)\},\\
&=\{(1,2,0),(2,0,0),(1,0,0),(2,1,0)\}.
\end{align*}
Thus $A$ is an $S_h$-linear for all $1\leq h\leq 3$, but clearly is a linearly dependent set.
\item Consider $A\subseteq\mathbb{F}_{2}^{10}$ given by
	\begin{equation*} 
		\begin{split}
			A=&\{\boldsymbol{e}_{1}, \boldsymbol{e}_{2}, \boldsymbol{e}_{10}, \boldsymbol{e}_{1}+\boldsymbol{e}_{3}, \boldsymbol{e}_{2}+\boldsymbol{e}_{4}, \boldsymbol{e}_{8}+\boldsymbol{e}_{9}, \boldsymbol{e}_{9}+\boldsymbol{e}_{10}, \boldsymbol{e}_{1}+\boldsymbol{e}_{3}+\boldsymbol{e}_{5},
			\boldsymbol{e}_{6}+\boldsymbol{e}_{7}+\boldsymbol{e}_{9},\\ 
			&\,\,\boldsymbol{e}_{7}+\boldsymbol{e}_{8}+\boldsymbol{e}_{10}, \boldsymbol{e}_{1}+\boldsymbol{e}_{2}+\boldsymbol{e}_{4}+\boldsymbol{e}_{6}, \boldsymbol{e}_{2}+\boldsymbol{e}_{3}+\boldsymbol{e}_{5}+\boldsymbol{e}_{7},
			\boldsymbol{e}_{3}+\boldsymbol{e}_{4}+\boldsymbol{e}_{6}+\boldsymbol{e}_{8},\\ 
			&\,\,\boldsymbol{e}_{4}+\boldsymbol{e}_{5}+\boldsymbol{e}_{7}+\boldsymbol{e}_{9}, \boldsymbol{e}_{5}+\boldsymbol{e}_{6}+\boldsymbol{e}_{8}+\boldsymbol{e}_{10} \}.
		\end{split}
	\end{equation*}
	It can be verified that $A$ is an $S_{3}$-linear set, but $A$ is a linearly dependent set, since that  
    \begin{align*}\boldsymbol{e}_{9}+\boldsymbol{e}_{10}=~&\boldsymbol{e}_{1}+\boldsymbol{e}_{2}+(\boldsymbol{e}_{1}+\boldsymbol{e}_{3})+(\boldsymbol{e}_{2}+\boldsymbol{e}_{4})+(\boldsymbol{e}_{3}+\boldsymbol{e}_{4}+\boldsymbol{e}_{6}+\boldsymbol{e}_{8})\\
    &+(\boldsymbol{e}_{6}+\boldsymbol{e}_{7}+\boldsymbol{e}_{9})+(\boldsymbol{e}_{7}+\boldsymbol{e}_{8}+\boldsymbol{e}_{10}).
    \end{align*}
However, $A$ is not an $S_4$-linear set, because 
\begin{align*}
\boldsymbol{e}_{1}+\boldsymbol{e}_{2}+\boldsymbol{e}_{10}+(\boldsymbol{e}_{8}+\boldsymbol{e}_{9})=~&(\boldsymbol{e}_{1}+\boldsymbol{e}_{3})+(\boldsymbol{e}_{6}+\boldsymbol{e}_{7}+\boldsymbol{e}_{9})\\
&+(\boldsymbol{e}_{2}+\boldsymbol{e}_{3}+\boldsymbol{e}_{5}+\boldsymbol{e}_{7})+(\boldsymbol{e}_{5}+\boldsymbol{e}_{6}+\boldsymbol{e}_{8}+\boldsymbol{e}_{10}).
\end{align*}

\end{enumerate}
\end{ejem}

In the next, we give some properties of $S_h$-linear sets. 

\begin{prop}
Let $h\geq 2$ be an integer and $A$ a subset of $\mathbb{F}_q^n$. Then 

\begin{enumerate}[\normalfont(i)]
	\item \label{part1} $A$ is an $S_h$-linear set if and only if
    
    \begin{equation}
	|\overline{h}A|= \begin{cases}
(q-1)^h \binom {|A|}{h}, & \text{ if $\boldsymbol{0}\notin A$},\\[0.25cm]
	(q-1)^{h-1} \binom {|A|-1}{h-1}+(q-1)^h \binom {|A|-1}{h}, &  \text{ if $\boldsymbol{0}\in A$}.
	\end{cases}
	\label{eq:hCL}
\end{equation} 
	\item \label{part2} If $A$ is an $S_h$-linear set and $q\neq 2$, then $\overline{h}A \cap \overline{t}A=\emptyset$, for all $1\leq t\leq h-1$.

	\item \label{part3} If $A$ is an $S_h$-linear set, then 
	\begin{equation}
	|A|< \begin{cases}
	\dfrac{\sqrt[h]{q^nh!}}{q-1} + (h-1), & \text{ if $\boldsymbol{0}\notin A$},\\[0.25cm]
	\sqrt[h]{\dfrac{q^nh!}{(q-1)^{h-1}}} + (h-1), &  \text{ if $\boldsymbol{0}\in A$}.
	\end{cases}
	\label{eq:contention0}
\end{equation}
\end{enumerate}

\end{prop}
\begin{proof}\quad
\begin{enumerate}[\normalfont(i)]  
\item Suppose that $\boldsymbol{0}\notin A$. Note that to construct an element of $\overline{h}A$, we need to choose $h$ elements from $A$, this task can be done in $\binom {|A|}{h}$ different ways. Then for each one of these elements we take a non-zero coefficient; which can be done in $q-1$ different ways.

In the other hand, if $\boldsymbol{0}\in A$, then   in order to construct an element of $\overline{h}A$, one must take into account whether $\boldsymbol{0}$ participates in the $h$-linear combination or not. If it does, then the choice reduces to selecting $h-1$ elements from $A \setminus \{\boldsymbol{0}\}$, yielding $\binom{|A|-1}{h-1}$ possibilities, and assigning to each of them a non-zero coefficient, which can be done in $(q-1)$ ways. Otherwise, when $\boldsymbol{0}$ is not involved, we select $h$ elements from $A \setminus \{\boldsymbol{0}\}$, which gives $\binom{|A|-1}{h}$ possibilities, and again each element is assigned a non-zero coefficient in $(q-1)$ ways.
\item Note that if an $h$-linear combination is equal to a $t$-linear combination, we can complete the last one to be also an $h$-linear combination. Indeed, suppose that for some $\boldsymbol{a}_{i},\boldsymbol{b}_{i}\in A$ and $\lambda_{i},\gamma_{i}\in \mathbb{F}_q^*$,
$$\sum_{i=1}^h\lambda_{i}\boldsymbol{a}_{i}=\sum_{i=1}^t\gamma_{i}\boldsymbol{b}_{i}.$$
Since $q\neq 2$ for each $1\leq i\leq h-t$ we can find $\delta_i\in \mathbb{F}_q^*$ such that $\lambda_{i}+\delta_i\in\mathbb{F}_q^*$. Thus we obtain the equality
$$\sum_{i=1}^{h-t}(\lambda_{i}+\delta_i)\boldsymbol{a}_{i}+\sum_{i=h-t+1}^h\lambda_{i}\boldsymbol{a}_{i}=\sum_{i=1}^{h-t}\delta_i\boldsymbol{a}_{i}+\sum_{i=1}^t\gamma_{i}\boldsymbol{b}_{i}.$$
Therefore, we get two equal $h$-linear combinations from $A$, which is a contradiction.

\item Assume that $\boldsymbol{0}\notin A$. Then by item \eqref{part1}, we get that
\begin{align*}
|\overline{h}A|&=(q-1)^h \binom {|A|}{h}=\dfrac{(q-1)^h|A|!}{h!(|A|-h)!}\\
&=\dfrac{(q-1)^h(|A|-h+1)(|A|-h+2)\cdots |A|}{h!}\\
&>\dfrac{(q-1)^h(|A|-h+1)^h}{h!}.
\end{align*}
Since that $|\overline{h}A|\leq q^n$, we obtain that
\begin{align*}
\dfrac{(q-1)^h(|A|-h+1)^h}{h!}<q^n\\
(|A|-h+1)^h<\dfrac{q^nh!}{(q-1)^h}\\
|A|-h+1<\sqrt[h]{\dfrac{q^nh!}{(q-1)^h}}\\
|A|<\dfrac{\sqrt[h]{q^nh!}}{q-1}+(h-1).
\end{align*}
Now, suppose that $\boldsymbol{0}\in A$, again by \eqref{part1} and the Pascal's rule, we obtain that 
\begin{align*}
q^n\geq |\overline{h}A|&=(q-1)^{h-1} \binom {|A|-1}{h-1}+(q-1)^h \binom {|A|-1}{h}\\
&>(q-1)^{h-1} \left( \binom {|A|-1}{h-1}+ \binom {|A|-1}{h}\right)=(q-1)^{h-1}\binom {|A|}{h}\\
&>\dfrac{(q-1)^{h-1}(|A|-h+1)^h}{h!},
\end{align*}
and the conclusion is obtained as in the previous case.
\end{enumerate}
\end{proof}
Observe that, when $q=h=2$, the items \eqref{part1} and \eqref{part3} do not depend on the fact that $\boldsymbol{0}$ is or not in $A$, i.e. there are obtained the same value for $|\overline{h}A|$ and the same bound for $|A|$, respectively. Furthermore, a better bound for the size of a $S_2$-set in $\mathbb{F}_2^n$ can be found in \cite[Proposition 2.1]{CP2024}. We recall that $\langle A \rangle$ denote the set of all linear combinations of elements from $A$.

\begin{prop}\label{prop:traslacion_afin}
Let $A$ be a non-empty subset of a finite vector space $V$ over $\mathbb{F}_q$, with $q\neq 2$. If $A$ is an $S_h$-linear set, then 
$\boldsymbol{v} +\alpha A $ is an $S_h$-linear set on $V$, for all $\boldsymbol{v}\in V\setminus \langle A \rangle$ and  $\alpha\in \mathbb{F}_q^*$.
\end{prop}
\begin{proof}
Assume that there are two equal $h$-linear combinations of $\boldsymbol{v} +\alpha A$, that is
\begin{equation}
\beta_{1}\boldsymbol{b}_{1}+\beta_{2}\boldsymbol{b}_{2}+\dots+\beta_{h}\boldsymbol{b}_{h}=\lambda_{1}\boldsymbol{c}_{1}+\lambda_{2}\boldsymbol{c}_{2}+\dots+\lambda_{h}\boldsymbol{c}_{h},
\label{eq:affine_proof}
\end{equation}
for some $ \boldsymbol{b}_{i},\boldsymbol{c}_{i}\in \boldsymbol{v} +\alpha A$ and $\lambda_{i}, \beta_{i}\in \mathbb{F}_q^*$. Since that, $\boldsymbol{b}_{i}=\boldsymbol{v}+\alpha \boldsymbol{a}_i$ and $\boldsymbol{c}_{i}=\boldsymbol{v}+\alpha \boldsymbol{a}'_i$, for some $\boldsymbol{a}_i,\boldsymbol{a}'_i\in A$, from \eqref{eq:affine_proof} we get 
\begin{equation}
\sum_{i=1}^h(\beta_i-\lambda_i)\boldsymbol{v}+ \beta'_{1}\boldsymbol{a}_{1}+\beta'_{2}\boldsymbol{a}_{2}+\cdots+\beta'_{h}\boldsymbol{a}_{h}=\lambda'_{1}\boldsymbol{a}'_{1}+\lambda'_{2}\boldsymbol{a}'_{2}+\cdots+\lambda'_{h}\boldsymbol{a}'_{h},
\label{eq:affine_proof2}
\end{equation}
where $\beta_i'=\beta_i\alpha$ and $\lambda'_i=\lambda_i\alpha$. As by hypothesis, $\boldsymbol{v}\not\in \langle A \rangle$, we must have that \linebreak $\sum_{i=1}^h(\beta_i-\lambda_i)=0$. Then from \eqref{eq:affine_proof2}, we get a contradiction.
\end{proof}
In the next example, we show that the condition $\boldsymbol{v}\notin \langle A\rangle$ in the last proposition is necessary.

\begin{ejem}\quad 

\begin{enumerate}
\item Consider the $S_3$-linear set $A=\{\boldsymbol{a}_1,\boldsymbol{a}_2,\ldots, \boldsymbol{a}_{14}\}$ in $\mathbb{F}_{3}^{9}$
\begin{equation*} 
\begin{split}
A=\{&\boldsymbol{0},\, \boldsymbol{e}_1 ,\, \boldsymbol{e}_8,\, \boldsymbol{e}_9,\, 2\boldsymbol{e}_1+\boldsymbol{e}_2, \,\boldsymbol{e}_7+2\boldsymbol{e}_9,\, 2\boldsymbol{e}_1+2\boldsymbol{e}_2+\boldsymbol{e}_3,\, 2\boldsymbol{e}_2+2\boldsymbol{e}_3+\boldsymbol{e}_4,\\ 
& \boldsymbol{e}_1+2\boldsymbol{e}_3+2\boldsymbol{e}_4+\boldsymbol{e}_5, \,\boldsymbol{e}_2+2\boldsymbol{e}_4+2\boldsymbol{e}_5+\boldsymbol{e}_6,\, \boldsymbol{e}_3+2\boldsymbol{e}_5+2\boldsymbol{e}_6+\boldsymbol{e}_7, \\ 
&\, \boldsymbol{e}_6+2\boldsymbol{e}_8+\boldsymbol{e}_9,\, \boldsymbol{e}_4+2\boldsymbol{e}_6+2\boldsymbol{e}_7+\boldsymbol{e}_8, \,\boldsymbol{e}_5+2\boldsymbol{e}_7+2\boldsymbol{e}_8+\boldsymbol{e}_9
\}.
\end{split}
\end{equation*} Take $\boldsymbol{v}=\boldsymbol{a}_{4}+\boldsymbol{a}_{6}+\boldsymbol{a}_{8}\in\langle A\rangle$. However, the set $\boldsymbol{v}+A$ is not an $S_3$-linear set because, the next two $3$-linear combinations from $\boldsymbol{v}+A$ are equal,
\begin{equation*}
	\boldsymbol{v}+(\boldsymbol{v}+\boldsymbol{a}_{4})+(\boldsymbol{v}+\boldsymbol{a}_{6})=2(\boldsymbol{v}+\boldsymbol{a}_{4})+2(\boldsymbol{v}+\boldsymbol{a}_{6})+(\boldsymbol{v}+\boldsymbol{a}_{8}).
\end{equation*}
Here, we use the fact that $\boldsymbol{0}\in A$, to see that $\boldsymbol{v}\in \boldsymbol{v}+A$. 
\item Now, we give an $S_{3}$-linear set which not contains the zero vector of $V$. Let be $B=\{\boldsymbol{b}_1,\boldsymbol{b}_2,\ldots,\boldsymbol{b}_8\}\subset \mathbb{F}_{5}^{12}$ given by
\begin{equation*}
	\begin{split}
	B =\{& \boldsymbol{e}_{1}+3\boldsymbol{e}_{2}+4\boldsymbol{e}_{3}+\boldsymbol{e}_{6},\,\boldsymbol{e}_{2}+3\boldsymbol{e}_{3}+4\boldsymbol{e}_{4}+2\boldsymbol{e}_{6},\,\boldsymbol{e}_{3}+3\boldsymbol{e}_{4}+4\boldsymbol{e}_{5}+2\boldsymbol{e}_{7},\\
	& \boldsymbol{e}_{4}+3\boldsymbol{e}_{5}+4\boldsymbol{e}_{6}+2\boldsymbol{e}_{8},\,\boldsymbol{e}_{5}+3\boldsymbol{e}_{6}+4\boldsymbol{e}_{7}+2\boldsymbol{e}_{9},\,\boldsymbol{e}_{6}+3\boldsymbol{e}_{7}+4\boldsymbol{e}_{8}+2\boldsymbol{e}_{10},\\
	&\boldsymbol{e}_{7}+3\boldsymbol{e}_{8}+4\boldsymbol{e}_{9}+2\boldsymbol{e}_{11},\, \boldsymbol{e}_{8}+3\boldsymbol{e}_{9}+4\boldsymbol{e}_{10}+2\boldsymbol{e}_{12} \}.
	\end{split}
\end{equation*}
For 
$\boldsymbol{u}=\boldsymbol{b}_{2}+2\boldsymbol{b}_{3}+\boldsymbol{b}_{4}\in\langle B\rangle$, the set $\boldsymbol{u} + B$ is not $S_{3}$-linear because the following $3$-linear combinations from $\boldsymbol{u} + B$ are equal,  
\begin{equation*}
(\boldsymbol{u}+\boldsymbol{b}_{2})+(\boldsymbol{u}+\boldsymbol{b}_{3})+(\boldsymbol{u}+\boldsymbol{b}_{4})=2(\boldsymbol{u}+\boldsymbol{b}_{2})+3(\boldsymbol{u}+\boldsymbol{b}_{3})+2(\boldsymbol{u}+\boldsymbol{b}_{4}).
\end{equation*}
\end{enumerate}
\end{ejem}

\begin{lem}\label{cojuntosh-1}
If $A$ is an $S_h$-linear set of a finite vector space of dimension $r$ over $\mathbb{F}_q$, where $2h < r \leq|A|$, then $A$ is an $S_j$-linear set for all $1 \leq j \leq h-1$.
\end{lem}

\begin{proof}
Suppose $A$ is not an $S_{j}$-linear set, then there exist at least two $j$-linear combinations of distinct elements of $A$ that are equal in $V$, i.e.,
\begin{equation}\label{eq:sh_1sums}
\lambda_{1}\boldsymbol{a}_{1}+\lambda_{2}\boldsymbol{a}_{2}+\dots+\lambda_{j}\boldsymbol{a}_{j}=\beta_{1}\boldsymbol{b}_{1}+\beta_{2}\boldsymbol{b}_{2}+\dots+\beta_{j}\boldsymbol{b}_{j},
\end{equation}
where $\lambda_{i},\beta_{i}\in \mathbb{F}_q^*$ and $\boldsymbol{a}_{i},\boldsymbol{b}_{i}\in A$ for all $1\leq i\leq j$. These linear combinations involve at most $2h-2$ elements from $A$, which is possible by hypothesis.

Now, if we add to both sides in \eqref{eq:sh_1sums} $h-j$ elements from $A$ that no appear in \eqref{eq:sh_1sums}, we obtain two equal $h$-linear combinations in $A$, this contradicts the assumption that $A$ is an $S_{h}$-linear set.
\end{proof}
Note that when $q=2$, the hypothesis that $\boldsymbol{v}\notin \langle A\rangle$ in Proposition \ref{prop:traslacion_afin} is not necessary, i.e. if $A$ is an $S_h$-linear set in a finite vector space over $\mathbb{F}_2$, then $\boldsymbol{v}+A$ is also an $S_h$-linear set for all $\boldsymbol{v}\in V$. Thus, given an $S_h$-linear set, we can construct an $S_h$-linear set that contains the zero vector. To obtain an analogous result for $q \neq 2$, we proceed as follows.

\begin{lem}\label{lemaAU0}
	Let $A$ be a non-empty subset of a finite vector space $V$ of dimension $r$ over $\mathbb{F}_q$, with $q \neq 2$. If $A$ is an $S_{h}$-linear set, where $2h < r \leq |A|$, then $A \cup \{\boldsymbol{0}\}$ is also an $S_{h}$-linear set.
\end{lem}

\begin{proof}
	If $\boldsymbol{0} \in A$, the result is immediate. Suppose $\boldsymbol{0} \notin A$ and that there are two equal $h$-linear combinations in $A \cup \{\boldsymbol{0}\}$, that is, 
	\begin{equation}
\beta_{1}\boldsymbol{b}_{1} + \beta_{2}\boldsymbol{b}_{2} + \dots + \beta_{h}\boldsymbol{b}_{h}=\lambda_{1}\boldsymbol{c}_{1} + \lambda_{2}\boldsymbol{c}_{2} + \dots + \lambda_{h}\boldsymbol{c}_{h},
		\label{eq:lemaAU0}	
	\end{equation}
	where $\beta_{i}, \lambda_{i} \in \mathbb{F}_{q}^*$ and $\boldsymbol{b}_{i}, \boldsymbol{c}_{i} \in A \cup \{\boldsymbol{0}\}$, for $1 \leq i \leq h$.
	
	Take $B = \{\boldsymbol{b}_{1}, \boldsymbol{b}_{2}, \ldots, \boldsymbol{b}_{h}\}$ and $C = \{\boldsymbol{c}_{1}, \boldsymbol{c}_{2}, \ldots, \boldsymbol{c}_{h}\}$.  We study the following cases:
	\begin{enumerate}
		\item If $\boldsymbol{0} \notin B \cup C$, then \eqref{eq:lemaAU0} is impossible because $A$ is $S_{h}$-linear. 
        
        \item If $\boldsymbol{0} \in B\cap C$, we get two equal $(h-1)$-linear combinations, but it is impossible by Lemma \ref{cojuntosh-1}.
		\item Without loss of generality, we can assume that $\boldsymbol{0}=\boldsymbol{b}_1 \in B$ and $\boldsymbol{0} \notin C$.
\begin{enumerate}
\item Let us suppose that  $B \cap C \neq \emptyset$. Assume, without restriction, that $|B\cap C|=1$ and $\boldsymbol{c}_{1}=\boldsymbol{b}_2 \in B\cap C$. Then from \eqref{eq:lemaAU0} we get 
		\begin{equation}
 \beta_{3}\boldsymbol{b}_{3} + \cdots + \beta_{h}\boldsymbol{b}_{h}=(\lambda_{1}-\beta_2)\boldsymbol{c}_{1} + \lambda_{2}\boldsymbol{c}_{2} + \dots + \lambda_{h}\boldsymbol{c}_{h}.
			\label{eq:lemaAU01}	
		\end{equation}
		If $\lambda_{1}\neq \beta_{2} $, the equation \eqref{eq:lemaAU01} gives rise to two equal $(h-1)$-linear combinations, which is impossible because by Lemma \ref{cojuntosh-1} $A$ is also an $S_{h-1}$-linear set. Now assume that, $\lambda_{1}=\beta_{2}$. Since that $|B^*| = h-1$, $|C| = h$ and $\boldsymbol{0} \notin C$, we obtain that $C \nsubseteq B$.  Take $\boldsymbol{c}_{k} \in C \setminus B$, for some $k\geq 2$. As $q\neq 2$, we can find $\delta \in \mathbb{F}_q^*$ such that $\lambda_k+\delta\neq 0$. Now, add $\delta \boldsymbol{c}_{k}$ to both sides of \eqref{eq:lemaAU01} to obtain
\begin{equation}
		\delta \boldsymbol{c}_{k} + \beta_{3}\boldsymbol{b}_{3} + \cdots + \beta_{h}\boldsymbol{b}_{h}=\lambda_{2}\boldsymbol{c}_{2} + \cdots + (\lambda_k+\delta)\boldsymbol{c}_k+\cdots+\lambda_{h}\boldsymbol{c}_{h}.
\label{eq:lemaAU02}	
\end{equation}
However, \eqref{eq:lemaAU02} contradicts that $A$ is an $S_{h-1}$-linear set.
		
		\item Secondly, suppose $B \cap C = \emptyset$. Then, since there exists $\gamma \in \mathbb{F}_{q}^{*}$ such that $\lambda_{1} + \gamma \neq 0$, we can add $\gamma \boldsymbol{c}_{1}$ to both sides of \eqref{eq:lemaAU0} to obtain
		$$\gamma\boldsymbol{c}_{1} + \beta_{2}\boldsymbol{b}_{2} + \cdots + \beta_{h}\boldsymbol{b}_{h}=(\lambda_{1}+\gamma)\boldsymbol{c}_{1} + \lambda_{2}\boldsymbol{c}_{2} + \dots + \lambda_{h}\boldsymbol{c}_{h},$$
		again leading to a contradiction.
		\end{enumerate}
	\end{enumerate}
	Thus, any pair of $h$-linear combinations in $A \cup \{\boldsymbol{0}\}$ are distinct. Therefore, $A \cup \{\boldsymbol{0}\}$ is an $S_{h}$-linear set.
\end{proof}

\begin{lem}\label{cojunto-sh-LI}
If $A$ is an $S_h$-linear set in $\mathbb{F}^{r}_{q}$ with $\boldsymbol{0} \in A$, where $2h <r\leq |A|$, then every subset of $A$ with $2h$ non-zero elements is linearly independent in $\mathbb{F}^{r}_{q}$.
\end{lem}
\begin{proof} Suppose $\{\boldsymbol{a}_{1},\boldsymbol{a}_{2},\dots,\boldsymbol{a}_{2h}\}\subset A$ is a linearly dependent set, where for all $1\leq i\leq 2h$ $\boldsymbol{a}_i\neq \boldsymbol{0}$. Then, there exist scalars $\lambda_{1},\lambda_{2},\ldots,\lambda_{2h}\in \mathbb{F}_{q}$, not all zero, such that 
\begin{equation}
\lambda_{1}\boldsymbol{a}_{1}+\lambda_{2}\boldsymbol{a}_{2}+\dots+\lambda_{2h}\boldsymbol{a}_{2h}=\boldsymbol{0}.
\label{eq:lemaconjunto-sh-LI}
\end{equation}
Let $t=|\{i:\lambda_{i}=0\}|$. If $t$ is even, two $(h-\frac{t}{2})$-linear combinations can be formed equal to each other, and from them, two $h$-linear combinations equals can be constructed by adding $\frac{t}{2}$ distinct vectors on both sides taken from those with zero coefficients in \eqref{eq:lemaconjunto-sh-LI}.

If $t$ is odd, then in \eqref{eq:lemaconjunto-sh-LI} we add $\lambda_{2h+1}\boldsymbol{0}$ where $\lambda_{2h+1}\in\mathbb{F}_{q}^*$ and $\boldsymbol{0}\in A$, i.e.,
\begin{equation}
\lambda_{1}\boldsymbol{a}_{1}+\lambda_{2}\boldsymbol{a}_{2}+\dots+\lambda_{2h}\boldsymbol{a}_{2h}+\lambda_{2h+1}\boldsymbol{0}=\boldsymbol{0}.
\label{eq:lemaconjunto-sh-LI2}
\end{equation}
Then, in \eqref{eq:lemaconjunto-sh-LI2}, there are $(2h+1-t)$ non-zero coefficients, and we can form two equal $(h-\frac{t-1}{2})$-linear combinations. By adding $\frac{t-1}{2}$ distinct vectors from $A$ on both sides, taken from those with zero coefficients in \eqref{eq:lemaconjunto-sh-LI2}, we obtain two $h$-linear combinations equals. In any case, we get a contradiction.
\end{proof}

\begin{defin}[Maximal $S_{h}$-linear set]
Let $M$ be an $S_{h}$-linear set in a finite vector space $V$ over $\mathbb{F}_q$. We say that $M$ is a maximal $S_{h}$-linear set if, for every $S_{h}$-linear set $A$ such that $M \subseteq A \subseteq V$, we have that $M = A$.
\end{defin}
The next result follows directly from Lemma \ref{lemaAU0}.

\begin{cor}\label{cor_maximal_contiene0}
If $M$ is a maximal $S_{h}$-linear set in a finite vector space $V$ of dimension $r$ over $\mathbb{F}_q$, with $q \neq 2$ and $2h <r\leq |M|$, then $\boldsymbol{0}\in M$.
\end{cor}

\begin{lem}\label{cojunto-sh-base}
If $A$ is a maximal $S_h$-linear set in $\mathbb{F}^{r}_{q}$, where $2h <r\leq |A|$, 
then $A$ contains a basis of $\mathbb{F}^{r}_{q}$ as a vector space over $\mathbb{F}_{q}$.
\end{lem}
\begin{proof} Assume that $A$ is a maximal $S_h$-linear set and does not contain a basis of $\mathbb{F}^{r}_{q}$ over $\mathbb{F}_{q}$, then $A$ is not a spanning set of $\mathbb{F}^{r}_q$, that is $\langle A\rangle\neq \mathbb{F}^{r}_{q}$. Thus, there exists $\boldsymbol{v}\in \mathbb{F}^{r}_{q}$, which is not a linear combination of elements of $A$. Now, $B=A\cup\{\boldsymbol{v}\}$ is an $S_{h}$-linear set in $\mathbb{F}^{r}_{q}$. Indeed, suppose that there are two equal $h$-linear combinations from $B$.  If both $h$-linear combinations have $\boldsymbol{v}$ as a term, we obtain either two equal $(h-1)$-linear combinations of $A$ or $\boldsymbol{v} \in\langle A\rangle$. In any case, we obtain a contradiction.
 Thus, we conclude that only one of the two $h$-linear combination has $\boldsymbol{v}$ as a term. Hence, we can prove again that $\boldsymbol{v} \in\langle A\rangle$, which is a contradiction. However, $B$ to be an $S_{h}$-linear set contradicts the maximality of $A$ as an $S_{h}$-linear set. Therefore, $A$ must contain a basis of $\mathbb{F}^{r}_{q}$.
\end{proof}

The following is our main result, which establishes a one-to-one correspondence between $S_h$-linear sets and a family of $q$-linear codes.

\begin{teor}\label{principalteor}
There exists a $q$-linear code with parameters $[n,k,d]$ such that $d \geq 2h+1$ if and only if there exists an $S_h$-linear set with $n+1$ elements in $\mathbb{F}^{n-k}_{q}$, where $n-k \geq 2h$.
\end{teor}
\begin{proof} Let $H$ be an $r\times n$ parity-check matrix of an $[n, k, d]$-linear code with minimum distance $d\geq 2h+1$, where $r = n-k$. By Theorem \ref{cojunto-LI-H}, its $n$ columns are non-zero and distinct; otherwise, it would be possible to find a linear dependent set of $d-1$ columns of $H$. Let $A=\{\col_1(H),\ldots,\col_n(H)\}\cup\{\boldsymbol{0}\}$ be the set of columns of $H$ union with $\{\boldsymbol{0}\}$ in $\mathbb{F}^{r}_{q}$. Now, by hypothesis $2h+1 \leq d$ and Singleton bound, see Theorem \ref{CotaSingleton}, we have that 
\begin{equation*}
\begin{split}
k+(2h+1) &\leq k+ d\leq n+1\\
 2h+1  & \leq  n-k+1   \\
     2h  & \leq  n-k.
\end{split}
\end{equation*}
Thus, $A$ contains more than $2h$ elements. 

On the other hand, assume that there are two equal $h$-linear combinations in $A$, i.e., such that
\begin{equation}\label{eq:hipomaskx}
\sum_{i=1}^{h}\lambda_{i}\boldsymbol{a}_{i}=\sum_{i=1}^{h}\beta_{i}\boldsymbol{b}_{i},
\end{equation}
with $\lambda_{i},\beta_{i}\in\mathbb{F}_{q}^*$. Then,
$$\lambda_{1}\boldsymbol{a}_{1}+\lambda_{2}\boldsymbol{a}_{2}+\dots+\lambda_{h}\boldsymbol{a}_{h}-\beta_{1}\boldsymbol{b}_{1}-\beta_{2}\boldsymbol{b}_{2}-\dots-\beta_{h}\boldsymbol{b}_{h}=\boldsymbol{0}.$$

Note that in \eqref{eq:hipomaskx} some terms on the left may be equal to terms on the right, but this cannot happen for all of them. In other words, we would have $2h$ or less $(2h-1)$ columns of $H$ that form a linearly dependent set. Since that $d\geq 2h+1$, it is a contradiction to Theorem \ref{cojunto-LI-H}.  Therefore, $A$ is an $S_{h}$-linear set in $\mathbb{F}_{q}^{r}$.

Conversely, let $r = n-k$ and suppose that there exists an $S_h$-linear set with $n+1$ elements in $\mathbb{F}^{r}_{q}$, where $n > r \geq 2h$; such a set is contained in a subset $A$ of $\mathbb{F}^{r}_{q}$ that is a maximal $S_h$-linear set. We can assume that $\boldsymbol{0}\in A$. Indeed, if $q=2$, we consider the set $\boldsymbol{v}+A$ for some $\boldsymbol{v}\in A$, while if $q\neq 2$, by Corollary \ref{cor_maximal_contiene0}, we have that $\boldsymbol{0}\in A$. By Lemma \ref{cojunto-sh-base}, $ A$ contains a basis of $\mathbb{F}^{r}_{q}$ over $\mathbb{F}_{q}$. Let $H$ be the $r\times n$ matrix whose columns are $n$ non-zero elements of $A$, including a basis of $\mathbb{F}^{r}_{q}$ over $\mathbb{F}_{q}$. By  Lemma \ref{cojunto-sh-LI} and Theorem \ref{cojunto-LI-H}, the $q$-linear code $\mathcal{C}$ with parity-check matrix $H$ satisfies that $d(\mathcal{C}) \geq 2h + 1$. Furthermore,  by Theorem 9 in \cite[Ch. 1, Sec. 10]{MS}, since $H$ has rank $r$, then $\mathcal{C}$ has dimension $k = n-r$.
\end{proof}

\begin{teor}\label{principalteor2}
	Let $h \geq 2$ and $n,r$ be positive integers such that $n>r\geq 2h$. If $A$ is an $S_{h}$-linear set in $\mathbb{F}_{q}^{r}$ with $n$ non-zero elements, then the $q$-linear code whose parity-check matrix has the $n$ non-zero elements of $A$ as columns is an $[n,t,d \geq 2h+1]$-linear code with $n-r\leq t\leq n-2h$. Moreover, if $A$ is a maximal $S_{h}$-linear set, then $t = n - r$.
\end{teor}
\begin{proof}
	Suppose that $A$ is an $S_h$-linear set with $n$ non-zero elements in $\mathbb{F}_{q}^{r}$, where $n > r \geq 2h$. In fact, if $q=2$, we consider the set $B=\boldsymbol{v}+A$ for some $\boldsymbol{v}\in A$, while if $q\neq 2$, by Lemma \ref{lemaAU0}, the set $B = A \cup \{\boldsymbol{0}\}$ is also $S_h$-linear. In any case, $\boldsymbol{0}\in B$. Let $H$ be the matrix of size $r \times n$, where its columns are the $n$ non-zero elements of $B$. By Lemma \ref{cojunto-sh-LI} and Theorem \ref{cojunto-LI-H}, the code $\mathcal{C}$ with parity-check matrix $H$ has minimum distance $d \geq 2h + 1$. Moreover, if $s$ is the rank of $H$, then by Lemma \ref{cojunto-sh-LI} and Theorem 9 in \cite[Ch. 1, Sec. 10]{MS}, $2h\leq s\leq r$, thus the dimension of $\mathcal{C}$ is $n-s$ and satisfies that $n-r \leq n-s\leq n-2h$.
\end{proof}

\begin{ejem} We give some applications of theorems \ref{principalteor} and \ref{principalteor2}.
\begin{enumerate}
\item A way to construct codes of minimum distance at least $2h+1$ is with BCH codes. For instance, consider the BCH code over $\mathbb{F}_5$ with generator polynomial $g(x)=x^8 + 2x^7 + 2x^5 + 2x^4 + 2x^3 + x^2 + 2$. This is a [12,4,7]-linear code over $\mathbb{F}_5$ with parity-check matrix given by

$$H_1=\left(\begin{array}{rrrrrrrrrrrr}
1 & 3 & 4 & 0 & 2 & 0 & 0 & 0 & 0 & 0 & 0 & 0 \\
0 & 1 & 3 & 4 & 0 & 2 & 0 & 0 & 0 & 0 & 0 & 0 \\
0 & 0 & 1 & 3 & 4 & 0 & 2 & 0 & 0 & 0 & 0 & 0 \\
0 & 0 & 0 & 1 & 3 & 4 & 0 & 2 & 0 & 0 & 0 & 0 \\
0 & 0 & 0 & 0 & 1 & 3 & 4 & 0 & 2 & 0 & 0 & 0 \\
0 & 0 & 0 & 0 & 0 & 1 & 3 & 4 & 0 & 2 & 0 & 0 \\
0 & 0 & 0 & 0 & 0 & 0 & 1 & 3 & 4 & 0 & 2 & 0 \\
0 & 0 & 0 & 0 & 0 & 0 & 0 & 1 & 3 & 4 & 0 & 2
\end{array}\right).$$
Thus the set of the columns of $H_1\cup\{\boldsymbol{0}\}$ is an $S_3$-linear set in $\mathbb{F}_5^8$.
	\item The set of the columns of the matrix $$H_2=\left(\begin{array}{rrrrrrrrrrrrrr}
1 & 0 & 0 & 0 & 0 & 0 & 0 & 0 & 1 & 0 & 0 & 1 & 1 & 0 \\
0 & 1 & 0 & 0 & 0 & 0 & 0 & 0 & 0 & 1 & 0 & 1 & 1 & 0 \\
0 & 0 & 1 & 0 & 0 & 0 & 0 & 0 & 1 & 0 & 1 & 1 & 1 & 1 \\
0 & 0 & 0 & 1 & 0 & 0 & 0 & 0 & 1 & 1 & 0 & 0 & 1 & 1 \\
0 & 0 & 0 & 0 & 1 & 0 & 0 & 0 & 1 & 1 & 1 & 0 & 1 & 0 \\
0 & 0 & 0 & 0 & 0 & 1 & 0 & 0 & 0 & 1 & 1 & 0 & 0 & 1 \\
0 & 0 & 0 & 0 & 0 & 0 & 1 & 0 & 0 & 0 & 1 & 0 & 1 & 1 \\
0 & 0 & 0 & 0 & 0 & 0 & 0 & 1 & 0 & 0 & 0 & 1 & 1 & 1
\end{array}\right)$$ forms an $S_2$-set in $\mathbb{F}_2^8$. Note that rank of $H_2$ is equal to 8, thus the binary code with parity-check matrix $H_2$ has parameters $[14,6,d]$ with $d\geq 5$. In fact, it can be verified that $d=5$.
\item The columns of the matrix $$H_3=\left(\begin{array}{rrrrrrrr}
0 & 0 & 0 & 0 & 0 & 1 & 1 & 0 \\
0 & 0 & 0 & 0 & 0 & 1 & 1 & 0 \\
0 & 0 & 0 & 0 & 1 & 1 & 1 & 1 \\
1 & 0 & 0 & 0 & 0 & 0 & 1 & 1 \\
0 & 1 & 0 & 0 & 1 & 0 & 1 & 0 \\
0 & 0 & 0 & 0 & 1 & 0 & 0 & 1 \\
0 & 0 & 1 & 0 & 1 & 0 & 1 & 1 \\
0 & 0 & 0 & 1 & 0 & 1 & 1 & 1
\end{array}\right)$$
form an $S_2$-set in $\mathbb{F}_2^8$. By Theorem \ref{principalteor2}, the binary code $\mathcal{C}_3$ whose parity-check matrix is $H_3$ has dimension $0\leq t\leq 4$. It can be verified that $\mathcal{C}_3$ is an $[8,2,5]$ binary code.
\end{enumerate}
\end{ejem}

By Theorem \ref{principalteor} and our discussion after Definition \ref{def:h_linear}, we obtain the next result.
\begin{cor}
If there exists an $[n, k, d]$-linear code $\mathcal{C}$ with  $d(\mathcal{C}) \geq 2h+1$, then there exists an $S_h$-set with $n+1$ elements in $\mathbb{F}^{n-k}_{q}$, where $n-k \geq 2h$.
\end{cor}

\section{Consequences of Theorem \ref{principalteor}} 

A basic problem in coding theory is to maximize the cardinal of a linear code $\mathcal{C}$ in $\mathbb{F}^{n}_{q}$ with minimum distance $d$, represented by the function
$$\mathcal{B}_{q}(n,d)=\max\left\{\left|\mathcal{C}\right|:  \mathcal{C}\subseteq\mathbb{F}_q^n \text{ is a $q$-linear code}, \text{ with } d(\mathcal{C})\geq d \right\}.$$
Recall that in $\mathbb{F}_{q}^{n}$ the cardinal of a linear code can be calculated as $\left|\mathcal{C}\right|=q^{k}$, where $\mathbb{F}_q$ is the dimension of $\mathcal{C}$ over $\mathbb{F}_{q}$ then analyzing the maximum cardinal  is equivalent to determining the maximum dimension of a code on $\mathbb{F}_{q}$; that is, $\log_{q}\mathcal{B}_{q}(n,d)$.

Now, from Theorem \ref{principalteor} we have that $q$-linear code of length $n$ and minimum distance $d\geq 2h+1$, with maximum dimension can be obtained by searching for the minimum redundancy $ r=n-k$ for which $\mathbb{F}_{q}^{r}$ has an  $S_{h}$-linear set with $n+1$ elements. This additive problem is presented with the following function
$$\overline{\mathcal{V}}_{q}(h,n)=\min\limits_{2h\leq r <n}\left\{r:\mathbb{F}_{q}^{r} \text{ contains an  } S_{h}\text{-linear set with } n+1 \text{ elements}\right\}.$$

The study of the function $\overline{\mathcal{V}}_{q}(h,n)$ is useful to calculate the function $\mathcal{B}_{q}(n,d)$, as the following consequence shows it.

\begin{cor}\label{cor:consequence_mainthm}
Let $n$ and $h$ be positive integers, such that $2h<n$.  Then
\begin{equation}
\log_{q}\mathcal{B}_{q}(n,2h+1)=n-\overline{\mathcal{V}}_{q}(h,n).
\label{eq:dimredu}
\end{equation}
\end{cor}

\begin{proof} Assume $r=\overline{\mathcal{V}}_{q}(h,n)$. Then there exists an $S_{h}$-linear set in $\mathbb{F}_{q}^{r}$ with $n+1$ elements such that $2h \leq r<n$. Thus, by  Theorem \ref{principalteor}, there exists an $[n,n-r,d]$-linear code $\mathcal{C}$ over $\mathbb{F}_{q}$ with $d(\mathcal{C})\geq 2h+1$, this implies that 
$$\log_{q}\mathcal{B}_{q}(n,2h+1)\geq n-r=n-\overline{\mathcal{V}}_{q}(h,n).$$
Now, suppose that there exists an $[n,k,d]$-linear code $\mathcal{C}$ with $d(\mathcal{C})\geq2h+1$ such that $k>n-\overline{\mathcal{V}}_{q}(h, n)$. Then, Theorem \ref{principalteor} guarantees the existence of an  $S_{h}$-linear set with $n+1$ elements in $\mathbb{F}_{q}^{n-k}$, where $n-k \geq 2h$. However,  $n-k <\overline{\mathcal{\mathcal{V}}}_{q}(h,n)$ contradicts the minimality of $\overline{\mathcal{V}}_{q}(h,n)$.
\end{proof}

Since that every $S_{h}$-linear set in $\mathbb{F}_{q}^{r}$ is also an $S_{h}$-set in $\mathbb{F}_{q}^{r}$, then  $\mathcal{V}_{q}(h,n)\leq\overline{\mathcal{V}}_{q}(h,n)$ where
$$\mathcal{V}_{q}(h,n)=\min\limits_{2h\leq r <n}\{r:\mathbb{F}_{q}^{r} \text{ contains an } S_{h}\text{-set  with } n+1 \text{ elements}\}.$$

From the above paragraph and Corollary \ref{cor:consequence_mainthm}, we have proven the next result. 

\begin{cor}
Let $n$ and $h$ be positive integers, such that $n>2h$. If $\mathcal{V}_{q}(h,n)$ exists, then
\begin{equation*}
\log_{q}\mathcal{B}_{q}(n,2h+1)\leq n-\mathcal{V}_{q}(h,n).
\label{eq:dimredu2}
\end{equation*}
\end{cor}

Recall that the concepts of $S_h$-linear set and $S_h$-set are equivalent when we are working in the binary case, thus Theorem \ref{principalteor} can be seen as a generalization of Theorem 1 in \cite{HO} and Theorem 6 in \cite{GT}. Furthermore, $\overline{\mathcal{V}}_{2}(h,n)=\mathcal{V}_{2}(h,n)$. 

Theorem \ref{principalteor} and the tables provided in \cite{Grassl:codetables} allows us to calculate $\mathcal{V}_{2}(h,n)$ for some values of $h$ and $n$, as we show in the sequel. Before presenting the examples, it is necessary to recall that the cardinal of the largest $S_{2}$-set in $\mathbb{F}_{2}^{4}$ is $6$, see \cite[Table 2]{HO}. Thus, $\mathbb{F}_{2}^{4}$ does not contain $S_{2}$-sets with $n+1$ elements, for $n\geq 6$. Furthermore, by Singleton bound, if $\mathcal{C}$ is an $[5,2]$-binary code, then $d(\mathcal{C})\leq 4$. Therefore, $\mathcal{V}_{2}(2,5)=4$.

Now, we calculate $\mathcal{V}_{2}(2,8)$. As in $\mathbb{F}_2^4$ the largest $S_2$-set has size $6$, we obtain that $\mathcal{V}_{2}(2,8)>4$. By Theorem \ref{principalteor}, we know that there exists an $S_{2}$-set
with $9$ elements in $\mathbb{F}_{2}^{r}$ if and only if there exists an $[8,k,d\geq 5]$-binary linear code such that $r=8-k$. From \cite{Grassl:codetables}, there is an $[8,2,5]$-binary linear code, thus $ \mathcal{V}_{2}(2,8)\leq 6$. But also, by searching in \cite{Grassl:codetables} there is no a binary linear code with parameters $[8,3,d\geq 5]$. Thus, $\mathcal{V}_{2}(2,8)=6$.

Similarly, we can calculate $\mathcal{V}_{2}(2,19)$. Again, by \cite{Grassl:codetables}, there is a $[19,10,5]$-binary linear code, hence $\mathcal{V}_{2}(2,19)\leq 9$. Besides, we can verify in \cite{Grassl:codetables} that there is no a $[19,k,d \geq 5]$-binary linear code such that $19 - k = r$, for $11\leq k\leq 14$. Thus, $\mathcal{V}_{2}(2,19)= 9$. 

Figure \ref{fig:nu2}, presents the values obtained for $\mathcal{V}_{2}(h,n)$, for $2\leq h\leq 6$ and $2h+1\leq n\leq 256$ using tables from \cite{Grassl:codetables}.  Also, Figure \ref{fig:nu3} shows the figure of $\overline{\mathcal{V}}_{3}(h,n)$ for $2\leq h \leq 6$ and $2h+1 \leq n\leq 243$. 

In general, if $A$ is an $S_h$-linear set in $\mathbb{F}_q^r$, then any non-empty subset of $A$ is also an $S_h$-linear set in $\mathbb{F}_q^r$. Consequently, if $m \geq n$, it follows that $\overline{\mathcal{V}}_{q}(h,n) \leq \overline{\mathcal{V}}_{q}(h,m)$.
Moreover, as illustrated in figures~\ref{fig:nu2} and~\ref{fig:nu3}, there exist values of $n$ for which 
$\overline{\mathcal{V}}_{q}(h,n+1) = \overline{\mathcal{V}}_{q}(h,n)$,
and values of $m$ for which 
$\overline{\mathcal{V}}_{q}(h,m+1) = \overline{\mathcal{V}}_{q}(h,m) + 1$.

\begin{figure}[ht]
	\centering
		\includegraphics[scale=0.8]{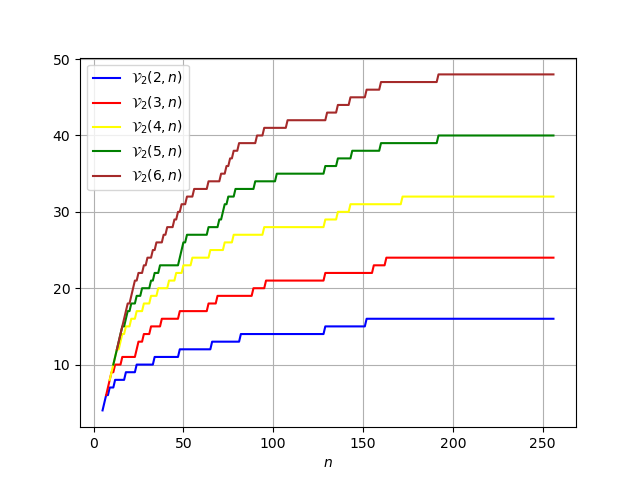}
	\caption{Values of $\mathcal{V}_2(h,n)$ for $2\leq h\leq 6$ and $2h+1\leq n\leq 256$.}
	\label{fig:nu2}
\end{figure}
\begin{figure}[ht]
	\centering
		\includegraphics[scale=0.8]{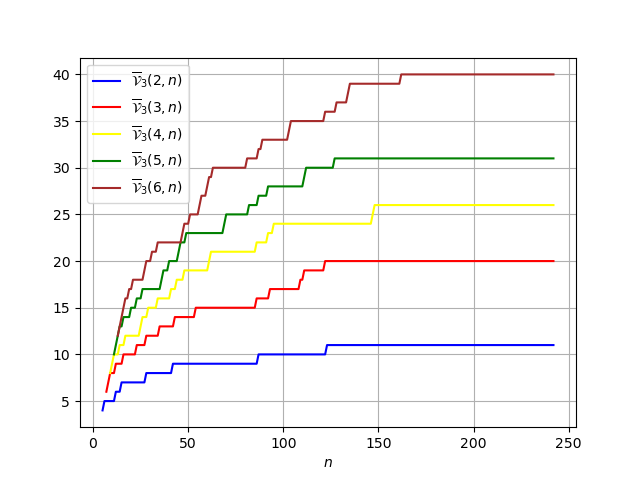}
	\caption{Values of $\overline{\mathcal{V}}_3(h,n)$ for $2\leq h\leq 6$ and $2h+1\leq n\leq 243$.}
	\label{fig:nu3}
\end{figure}

We now consider the function \begin{align*}
\overline{S}_h(\mathbb{F}_q^r)&=\max\left\{ |A|: A\subseteq \mathbb{F}_q^r \text{ is an $S_h$-linear set} \right\},
\end{align*} 
for some $h$, $q$ and $r$. Note that $\overline{S}_h(\mathbb{F}_q^r)\leq S_h(\mathbb{F}_q^r)$, where $S_h(\mathbb{F}_q^r)$ is the maximal cardinal of an $S_h$-set in $\mathbb{F}_q^r$. Also, $\overline{S}_h(\mathbb{F}_2^r)=S_h(\mathbb{F}_2^r)\geq h$ for all $h\leq 2^r$. 

The combination of Theorem \ref{principalteor} and the tables provided in \cite{Grassl:codetables} allows us to calculate lower bounds for $\overline{S}_h(\mathbb{F}_q^r)$ for $q=2,3,4,5,7$ and $9$. In the tables of \cite{Grassl:codetables}, for a fixed pair $(n,k)$, it is given lower and upper bounds for the minimum distance of an $[n,k]$ code. In some cases, these tables provide an exact value for the greatest possible value of this minimum distance. Thus, the procedure is as follows: given $r\in \mathbb{Z}^+$ and $h\geq 2$ search pairs $(n,k)$ such that $r=n-k$ and there exist an $[n,k,d]$-code with $d\geq 2h+1$ in \cite{Grassl:codetables}. For instance, for $r=9$ we can find that for the pair $(21,12)$ there is a $[21,12,5]$ binary code. Thus, by Theorem \ref{principalteor}, we know that there exists an $S_2$-set in $\mathbb{F}_2^9$ with 22 elements. Following this idea, we can iterate through the table in \cite{Grassl:codetables} searching for pairs $(n,k)$ such that $n-k=9$ and for which there is a $[n,k,5]$ binary code: $(19, 10), (20, 11), (21, 12), (22, 13), (23, 14)$. This means that there is an $S_2$-set with 24 elements in $\mathbb{F}_2^9$. Thus, $S_2(\mathbb{F}_2^9) \geq 24$. Some of these lower bounds are given in Table \ref{table:values_of_Sh_2}. Also, this compilation can be done in linear codes over $\mathbb{F}_3, \mathbb{F}_4, \mathbb{F}_5, \mathbb{F}_7$ and $\mathbb{F}_9$, see tables \ref{table:values_of_Sh_3}, \ref{table:values_of_Sh_5}, \ref{table:values_of_Sh_7}, \ref{table:values_of_Sh_8}, \ref{table:values_of_Sh_4} and \ref{table:values_of_Sh_9}.\\

Note that for certain values of $q$, $r$, and $h$, it does not make sense to compute $\overline{S}_h(\mathbb{F}_q^r)$ from our construction. This occurs because it is not possible to find a code with the required parameters. For instance, when $r=4$, there is no $[n,k]$-binary code with $n-k=4$ and $d\geq 7$, since, by the Singleton bound, we have $d\leq 5$ in this case. For such situations, we write X in the corresponding cell of the relevant table.

On the other hand, if we have a set $A$ which is an $S_h$-linear set in $\mathbb{F}_q^r$, it suffices to append a coordinate equal to zero to each of its elements to obtain an $S_h$-linear set in $\mathbb{F}_q^{r+1}$. This shows that $\overline{S}_h(\mathbb{F}_q^r) \leq \overline{S}_h(\mathbb{F}_q^t)$ for all $t \geq r$.

Unfortunately, there is a limit to the code lengths that can be studied based on \cite{Grassl:codetables}: for binary and quaternary codes, the maximum length is 256; for ternary codes, it is 243; for codes over $\mathbb{F}_5$, $\mathbb{F}_8$ and  $\mathbb{F}_9$, the maximum is 130 and finally for $\mathbb{F}_7$ is 100. Thus, for instance if there exists an $r$ such that $S_2(\mathbb{F}_2^r)\geq 257$, then by the observation made in the previous paragraph, for all $t \geq r$, our computations will yield the same lower bound; that is, $S_2(\mathbb{F}_2^t)\geq 257$. 

Finally, we note that deriving explicit expressions for the functions 
$\overline{\mathcal{V}}_{q}(h,n)$ and $\overline{S}_h(\mathbb{F}_q^r)$ 
remains an open problem that may lead to further research.

\section*{Acknowledgments}
Viviana Guerrero expresses her gratitude for the support of  MINCIENCIAS - Colombia for his doctoral studies through the ``Convocatoria del Fondo de Ciencia, Tecnología e Innovación del Sistema General de Regalías para la conformación de una lista de proyectos elegibles para ser viabilizados, priorizados y aprobados por el OCAD, en el marco del Programa de Becas de Excelencia Doctoral del Bicentenario Corte  BPIN: 2020000100319." J.H. Castillo was partially supported by Vicerrector\'{i}a de Investigaciones e Interacci\'on Social at Universidad de Nari\~no. C. Trujillo acknowledges the support of the Doctorado en Ciencias Matemáticas at the Universidad del Cauca. The authors are members of the research group ``Álgebra, Teoría de Números y Apliciones: ERM". ALTENUA is supported by Universidad del Cauca, Universidad de Antioquia, Universidad del Valle, and Universidad de Nariño.


\newpage 
\begin{table}[H]
\centering
\begin{tabular}{|c|c|c|c|c|c|c|c|}
\hline
$r$ & $S_2(\mathbb{F}_2^r)$ & $S_3(\mathbb{F}_2^r)$ & $S_4(\mathbb{F}_2^r)$ & $S_5(\mathbb{F}_2^r)$ & $S_6(\mathbb{F}_2^r)$ & $S_7(\mathbb{F}_2^r)$ & $S_8(\mathbb{F}_2^r)$ \\ \hline \hline 
4 & 6 & X & X & X & X & X & X \\ \hline  5 & 7 & X & X & X & X & X & X \\ \hline  6 & 9 & 8 & X & X & X & X & X \\ \hline  7 & 12 & 9 & X & X & X & X & X \\ \hline  8 & 18 & 10 & 10 & X & X & X & X \\ \hline  9 & 24 & 12 & 11 & X & X & X & X \\ \hline  10 & 34 & 16 & 12 & 12 & X & X & X \\ \hline  11 & 48 & 24 & 13 & 13 & X & X & X \\ \hline  12 & 66 & 25 & 15 & 14 & 14 & X & X \\ \hline  13 & 82 & 28 & 16 & 15 & 15 & X & X \\ \hline  14 & 129 & 32 & 18 & 16 & 16 & 16 & X \\ \hline  15 & 152 & 38 & 21 & 18 & 17 & 17 & X \\ \hline  16 & 257 & 48 & 24 & 19 & 18 & 18 & 18 \\ \hline  17 & 257 & 64 & 28 & 21 & 19 & 19 & 19 \\ \hline  18 & 257 & 69 & 32 & 24 & 21 & 20 & 20 \\ \hline  19 & 257 & 89 & 36 & 27 & 22 & 21 & 21 \\ \hline  20 & 257 & 96 & 42 & 32 & 23 & 22 & 22 \\ \hline  21 & 257 & 129 & 46 & 34 & 25 & 24 & 23 \\ \hline  22 & 257 & 156 & 50 & 37 & 28 & 25 & 24 \\ \hline  23 & 257 & 163 & 55 & 48 & 30 & 26 & 25 \\ \hline  24 & 257 & 257 & 65 & 49 & 33 & 28 & 27 \\ \hline  25 & 257 & 257 & 73 & 50 & 35 & 32 & 28 \\ \hline  26 & 257 & 257 & 78 & 52 & 39 & 33 & 29 \\ \hline  27 & 257 & 257 & 95 & 64 & 41 & 36 & 30 \\ \hline  28 & 257 & 257 & 129 & 70 & 45 & 38 & 32 \\ \hline  29 & 257 & 257 & 136 & 72 & 47 & 41 & 33 \\ \hline  30 & 257 & 257 & 143 & 73 & 49 & 45 & 35 \\ \hline  31 & 257 & 257 & 172 & 75 & 52 & 48 & 38 \\ \hline  32 & 257 & 257 & 257 & 79 & 56 & 49 & 40 \\ \hline  33 & 257 & 257 & 257 & 90 & 64 & 52 & 43 \\ \hline  \end{tabular}
\caption{Lower bounds for $S_h(\mathbb{F}_2^r)$ for $2\leq h\leq 8$ and $4\leq r \leq 33$.}
	 \label{table:values_of_Sh_2}
	 \end{table}

\begin{table}[H]
\centering
\begin{tabular}{|c|c|c|c|c|c|c|c|}\hline $r$ & $\overline{S}_2(\mathbb{F}_3^r)$ & $\overline{S}_3(\mathbb{F}_3^r)$ & $\overline{S}_4(\mathbb{F}_3^r)$ & $\overline{S}_5(\mathbb{F}_3^r)$ & $\overline{S}_6(\mathbb{F}_3^r)$ & $\overline{S}_7(\mathbb{F}_3^r)$ & $\overline{S}_8(\mathbb{F}_3^r)$ \\ \hline  \hline
 4 & 6 & X & X & X & X & X & X \\ \hline  5 & 12 & X & X & X & X & X & X \\ \hline  6 & 15 & 8 & X & X & X & X & X \\ \hline  7 & 28 & 9 & X & X & X & X & X \\ \hline  8 & 42 & 12 & 10 & X & X & X & X \\ \hline  9 & 87 & 16 & 11 & X & X & X & X \\ \hline  10 & 123 & 23 & 14 & 12 & X & X & X \\ \hline  11 & 244 & 28 & 17 & 13 & X & X & X \\ \hline  12 & 244 & 35 & 25 & 14 & 14 & X & X \\ \hline  13 & 244 & 43 & 26 & 16 & 15 & X & X \\ \hline  14 & 244 & 54 & 29 & 20 & 16 & 16 & X \\ \hline  15 & 244 & 86 & 34 & 23 & 17 & 17 & X \\ \hline  16 & 244 & 93 & 41 & 26 & 19 & 18 & 18 \\ \hline  17 & 244 & 109 & 44 & 36 & 21 & 19 & 19 \\ \hline  18 & 244 & 111 & 48 & 37 & 27 & 21 & 20 \\ \hline  19 & 244 & 122 & 61 & 40 & 28 & 23 & 21 \\ \hline  20 & 244 & 244 & 62 & 45 & 31 & 29 & 22 \\ \hline  21 & 244 & 244 & 86 & 46 & 34 & 30 & 24 \\ \hline  22 & 244 & 244 & 92 & 49 & 47 & 33 & 27 \\ \hline  23 & 244 & 244 & 95 & 69 & 48 & 34 & 29 \\ \hline  24 & 244 & 244 & 147 & 70 & 51 & 49 & 32 \\ \hline  25 & 244 & 244 & 148 & 82 & 56 & 50 & 36 \\ \hline  26 & 244 & 244 & 244 & 87 & 57 & 53 & 39 \\ \hline  27 & 244 & 244 & 244 & 92 & 60 & 58 & 40 \\ \hline  28 & 244 & 244 & 244 & 111 & 61 & 59 & 42 \\ \hline  29 & 244 & 244 & 244 & 112 & 63 & 62 & 60 \\ \hline  30 & 244 & 244 & 244 & 127 & 81 & 63 & 61 \\ \hline  31 & 244 & 244 & 244 & 244 & 87 & 64 & 64 \\ \hline  32 & 244 & 244 & 244 & 244 & 89 & 65 & 65 \\ \hline  33 & 244 & 244 & 244 & 244 & 103 & 70 & 66 \\ \hline \end{tabular}
\caption{Lower bounds for $\overline{S}_h(\mathbb{F}_3^r)$ for $2\leq h\leq 8$ and $4\leq r \leq 33$.}
\label{table:values_of_Sh_3}
\end{table}

\begin{table}[H]
\centering
\begin{tabular}{|c|c|c|c|c|c|c|c|}\hline
$r$ & $\overline{S}_2(\mathbb{F}_4^r)$ & $\overline{S}_3(\mathbb{F}_4^r)$ & $\overline{S}_4(\mathbb{F}_4^r)$ & $\overline{S}_5(\mathbb{F}_4^r)$ & $\overline{S}_6(\mathbb{F}_4^r)$ & $\overline{S}_7(\mathbb{F}_4^r)$ & $\overline{S}_8(\mathbb{F}_4^r)$ \\ \hline  \hline 
4 & 6 & X & X & X & X & X & X \\ \hline  5 & 12 & X & X & X & X & X & X \\ \hline  6 & 22 & 8 & X & X & X & X & X \\ \hline  7 & 44 & 10 & X & X & X & X & X \\ \hline  8 & 86 & 18 & 10 & X & X & X & X \\ \hline  9 & 172 & 22 & 11 & X & X & X & X \\ \hline  10 & 257 & 27 & 15 & 12 & X & X & X \\ \hline  11 & 257 & 43 & 19 & 13 & X & X & X \\ \hline  12 & 257 & 47 & 28 & 17 & 14 & X & X \\ \hline  13 & 257 & 71 & 29 & 18 & 15 & X & X \\ \hline  14 & 257 & 114 & 32 & 30 & 16 & 16 & X \\ \hline  15 & 257 & 123 & 43 & 31 & 19 & 17 & X \\ \hline  16 & 257 & 148 & 52 & 33 & 22 & 18 & 18 \\ \hline  17 & 257 & 257 & 66 & 36 & 25 & 21 & 19 \\ \hline  18 & 257 & 257 & 69 & 40 & 28 & 24 & 20 \\ \hline  19 & 257 & 257 & 88 & 43 & 32 & 26 & 21 \\ \hline  20 & 257 & 257 & 112 & 65 & 35 & 30 & 23 \\ \hline \end{tabular}
\caption{Lower bounds for $\overline{S}_h(\mathbb{F}_4^r)$ for $2\leq h\leq 8$ and $4\leq r \leq 20$.}
\label{table:values_of_Sh_4}
\end{table}

\begin{table}[H]
\centering
\begin{tabular}{|c|c|c|c|c|c|c|c|}\hline 
$r$ & $\overline{S}_2(\mathbb{F}_5^r)$ & $\overline{S}_3(\mathbb{F}_5^r)$ & $\overline{S}_4(\mathbb{F}_5^r)$ & $\overline{S}_5(\mathbb{F}_5^r)$ & $\overline{S}_6(\mathbb{F}_5^r)$ & $\overline{S}_7(\mathbb{F}_5^r)$ & $\overline{S}_8(\mathbb{F}_5^r)$ \\ \hline   
4 & 7 & X & X & X & X & X & X \\ \hline  5 & 13 & X & X & X & X & X & X \\ \hline  6 & 31 & 8 & X & X & X & X & X \\ \hline  7 & 45 & 12 & X & X & X & X & X \\ \hline  8 & 127 & 18 & 10 & X & X & X & X \\ \hline  9 & 131 & 28 & 12 & X & X & X & X \\ \hline  10 & 131 & 34 & 16 & 12 & X & X & X \\ \hline  11 & 131 & 46 & 22 & 13 & X & X & X \\ \hline  12 & 131 & 64 & 28 & 17 & 14 & X & X \\ \hline  13 & 131 & 126 & 33 & 20 & 15 & X & X \\ \hline  14 & 131 & 130 & 37 & 30 & 17 & 16 & X \\ \hline  15 & 131 & 131 & 49 & 31 & 21 & 17 & X \\ \hline  16 & 131 & 131 & 63 & 33 & 26 & 19 & 18 \\ \hline  17 & 131 & 131 & 67 & 37 & 27 & 22 & 19 \\ \hline  18 & 131 & 131 & 79 & 41 & 32 & 26 & 20 \\ \hline  19 & 131 & 131 & 131 & 48 & 35 & 28 & 24 \\ \hline  20 & 131 & 131 & 131 & 64 & 42 & 31 & 26 \\ \hline 
\end{tabular}
\caption{Lower bounds for $\overline{S}_h(\mathbb{F}_5^r)$ for $2\leq h\leq 8$ and $4\leq r \leq 20$.}
\label{table:values_of_Sh_5}
\end{table}

\begin{table}[H]
\centering
\begin{tabular}{|c|c|c|c|c|c|c|c|}\hline
$r$ & $\overline{S}_2(\mathbb{F}_7^r)$ & $\overline{S}_7(\mathbb{F}_7^r)$ & $\overline{S}_4(\mathbb{F}_7^r)$ & $\overline{S}_5(\mathbb{F}_7^r)$ & $\overline{S}_6(\mathbb{F}_7^r)$ & $\overline{S}_7(\mathbb{F}_7^r)$ & $\overline{S}_8(\mathbb{F}_7^r)$ \\ \hline  \hline  
4 & 9 & X & X & X & X & X & X \\ \hline  5 & 19 & X & X & X & X & X & X \\ \hline  6 & 45 & 9 & X & X & X & X & X \\ \hline  7 & 71 & 15 & X & X & X & X & X \\ \hline  8 & 101 & 22 & 10 & X & X & X & X \\ \hline  9 & 101 & 28 & 14 & X & X & X & X \\ \hline  10 & 101 & 42 & 21 & 12 & X & X & X \\ \hline  11 & 101 & 56 & 24 & 15 & X & X & X \\ \hline  12 & 101 & 101 & 29 & 19 & 14 & X & X \\ \hline  13 & 101 & 101 & 50 & 22 & 16 & X & X \\ \hline  14 & 101 & 101 & 53 & 31 & 20 & 16 & X \\ \hline  15 & 101 & 101 & 60 & 32 & 23 & 17 & X \\ \hline  16 & 101 & 101 & 101 & 52 & 33 & 21 & 18 \\ \hline  17 & 101 & 101 & 101 & 54 & 34 & 24 & 19 \\ \hline  18 & 101 & 101 & 101 & 58 & 37 & 26 & 22 \\ \hline  19 & 101 & 101 & 101 & 60 & 40 & 29 & 25 \\ \hline  20 & 101 & 101 & 101 & 101 & 54 & 33 & 28 \\ \hline  \end{tabular}\caption{Lower bounds for $\overline{S}_h(\mathbb{F}_7^r)$ for $2\leq h\leq 8$ and $4\leq r \leq 20$.}
\label{table:values_of_Sh_7}
\end{table}

\begin{table}[H]
\centering
\begin{tabular}{|c|c|c|c|c|c|c|c|}\hline
$r$ & $\overline{S}_2(\mathbb{F}_8^r)$ & $\overline{S}_3(\mathbb{F}_8^r)$ & $\overline{S}_4(\mathbb{F}_8^r)$ & $\overline{S}_5(\mathbb{F}_8^r)$ & $\overline{S}_6(\mathbb{F}_8^r)$ & $\overline{S}_7(\mathbb{F}_8^r)$ & $\overline{S}_8(\mathbb{F}_8^r)$ \\ \hline  \hline
4 & 10 & X & X & X & X & X & X \\ \hline  5 & 21 & X & X & X & X & X & X \\ \hline  6 & 59 & 10 & X & X & X & X & X \\ \hline  7 & 82 & 16 & X & X & X & X & X \\ \hline  8 & 131 & 25 & 10 & X & X & X & X \\ \hline  9 & 131 & 32 & 15 & X & X & X & X \\ \hline  10 & 131 & 75 & 21 & 12 & X & X & X \\ \hline  11 & 131 & 76 & 25 & 16 & X & X & X \\ \hline  12 & 131 & 131 & 38 & 20 & 14 & X & X \\ \hline  13 & 131 & 131 & 43 & 25 & 16 & X & X \\ \hline  14 & 131 & 131 & 64 & 30 & 21 & 16 & X \\ \hline  15 & 131 & 131 & 74 & 36 & 25 & 18 & X \\ \hline  16 & 131 & 131 & 131 & 66 & 28 & 21 & 18 \\ \hline  17 & 131 & 131 & 131 & 67 & 33 & 25 & 19 \\ \hline  18 & 131 & 131 & 131 & 70 & 37 & 29 & 23 \\ \hline  19 & 131 & 131 & 131 & 75 & 40 & 31 & 26 \\ \hline  20 & 131 & 131 & 131 & 131 & 68 & 36 & 28 \\ \hline \end{tabular}\caption{Lower bounds for $\overline{S}_h(\mathbb{F}_8^r)$ for $2\leq h\leq 8$ and $4\leq r \leq 20$.}
\label{table:values_of_Sh_8}
\end{table}

\begin{table}[H]
\centering
\begin{tabular}{|c|c|c|c|c|c|c|c|}\hline
$r$ & $\overline{S}_2(\mathbb{F}_9^r)$ & $\overline{S}_3(\mathbb{F}_9^r)$ & $\overline{S}_4(\mathbb{F}_9^r)$ & $\overline{S}_5(\mathbb{F}_9^r)$ & $\overline{S}_6(\mathbb{F}_9^r)$ & $\overline{S}_7(\mathbb{F}_9^r)$ & $\overline{S}_8(\mathbb{F}_9^r)$ \\ \hline  \hline 
4 & 11 & X & X & X & X & X & X \\ \hline  5 & 21 & X & X & X & X & X & X \\ \hline  6 & 73 & 11 & X & X & X & X & X \\ \hline  7 & 97 & 18 & X & X & X & X & X \\ \hline  8 & 131 & 23 & 11 & X & X & X & X \\ \hline  9 & 131 & 42 & 20 & X & X & X & X \\ \hline  10 & 131 & 53 & 21 & 12 & X & X & X \\ \hline  11 & 131 & 88 & 29 & 17 & X & X & X \\ \hline  12 & 131 & 131 & 42 & 21 & 14 & X & X \\ \hline  13 & 131 & 131 & 43 & 29 & 17 & X & X \\ \hline  14 & 131 & 131 & 61 & 31 & 21 & 16 & X \\ \hline  15 & 131 & 131 & 90 & 41 & 29 & 18 & X \\ \hline  16 & 131 & 131 & 131 & 42 & 31 & 22 & 18 \\ \hline  17 & 131 & 131 & 131 & 82 & 34 & 29 & 20 \\ \hline  18 & 131 & 131 & 131 & 85 & 40 & 31 & 23 \\ \hline  19 & 131 & 131 & 131 & 92 & 45 & 34 & 29 \\ \hline  20 & 131 & 131 & 131 & 131 & 84 & 37 & 31 \\ \hline   \end{tabular}\caption{Lower bounds for $\overline{S}_h(\mathbb{F}_9^r)$ for $2\leq h\leq 8$ and $4\leq r \leq 20$.}
\label{table:values_of_Sh_9}
\end{table}

\end{document}